\documentclass[12pt]{amsart}
\usepackage[utf8]{inputenc}
\usepackage[T1]{fontenc}
\usepackage{amsfonts, accents}
\usepackage[leqno]{amsmath}
\usepackage{amssymb, comment, float}
\usepackage[dvipsnames]{xcolor}
\usepackage[foot]{amsaddr}
\usepackage[shortlabels]{enumitem}
\usepackage{mathdots}
\usepackage[a4paper, footskip=0.5in, headheight = 0.5in, top=1.25in, bottom=1.25in, right=1in, left=1in]{geometry}
\usepackage{hyperref}
\hypersetup{
colorlinks,
linkcolor=black,
urlcolor=Blue,
citecolor=black,}
\usepackage[center]{caption}
\usepackage{eufrak}
\usepackage{marvosym}     
\usepackage{latexsym}
\usepackage[nodayofweek]{datetime}
\usepackage{tikz}
\usepackage{subfig}
\usepackage{thm-restate}
\usepackage[none]{hyphenat} 

\newtheorem{theorem}{Theorem}[section]

\newtheorem{corollary}[theorem]{Corollary}

\newcommand{\mca}{\mathcal}
\newcommand{\poi}{\mathbb{N}}

\newcounter{tbox}
\newcommand{\sta}[1]{\medskip\medskip\refstepcounter{tbox}\noindent{\parbox{\textwidth}{(\thetbox) \emph{#1}}}\vspace*{0.3cm}}
\newcommand{\mylongtitle}[1]{%
  \ifodd\value{page}%
    \protect\parbox{0.97\linewidth}{#1}\hfill%
  \else%
    \hfill\protect\parbox{0.97\linewidth}{#1}%
  \fi%
}

\title[Bull-free graphs and $\chi$-boundedness]{Bull-free graphs and $\chi$-boundedness}

\author{Sepehr Hajebi$^{\dagger}$}

\thanks{$^{\dagger}$ Department of Combinatorics and Optimization, University of Waterloo, Waterloo, Ontario, Canada}
\date{\today}
\begin{document}
\maketitle

\begin{abstract}
 A \textit{bull} is a graph obtained from a four-vertex path by adding a vertex adjacent to the two middle vertices of the path. A graph $G$ is \textit{bull-free} if no induced subgraph of $G$ is a bull. We prove that for all $k,t\in \poi$, if $G$ is a bull-free graph of clique number at most $k$ and every triangle-free induced subgraph of $G$ has chromatic number at most $t$, then $G$ has chromatic number at most $k^{O(\log t)}$. We further show that the bound $k^{O(\log t)}$ is best possible up to a multiplicative constant in the exponent.

Thomass\'{e}, Trotignon, and Vu\v{s}kovi\'{c} (2017) were the first to give a bound of the form $2^{p\log p}$, where $p=O(k^2+t)$, with a proof that uses Chudnovsky's structure theorem for bull-free graphs. This was improved by Chudnovsky, Cook, Davies, and Oum (2026) to a bound of the form $k^{O(t)}$, with a 10-page proof that again relies heavily on Chudnovsky's structure theorem.

Our proof is a single page long and completely avoids the structure theorem, instead using only a result of Chudnovsky and Safra (which itself has a short proof).
\end{abstract}

\section{Introduction}
We denote the set of all positive integers by $\poi$. Graphs in this paper have finite vertex sets, no loops, and no parallel edges. Given a graph $G=(V(G),E(G))$, the chromatic number and the clique number of $G$ are denoted, respectively, by $\chi(G)$ and $\omega(G)$. A graph class $\mca{C}$ is \textit{hereditary} if $\mca{C}$ is closed under taking induced subgraphs, and $\mca{C}$ is \textit{{\rm(}polynomially{\rm)} $\chi$-bounded} if there is a (polynomial) function $f:\poi\rightarrow\poi$ such that $\chi(G)\leq f(\omega(G))$ for every $G\in \mca{C}$. The study of $\chi$-bounded classes is a central topic in structural graph theory; see \cite{scotticm, survey} for surveys. This work focuses on $\chi$-boundedness in the class of bull-free graphs. A \textit{bull} is a graph obtained from a four-vertex path by adding a vertex adjacent to the two middle vertices of the path. A graph $G$ is \textit{bull-free} if no induced subgraph of $G$ is a bull. 

The class of bull-free graphs is not $\chi$-bounded since triangle-free graphs are bull-free, and there are triangle-free graphs of arbitrarily large chromatic number \cite{mycielski}. Thomass\'{e}, Trotignon, and Vu\v{s}kovi\'{c} \cite{tripaper} proved that the latter is in fact the only obstruction to $\chi$-boundedness for bull-free graphs. For all $t\in \poi$, let $\mca{B}_{t}$ be the class of all bull-free graphs whose triangle-free induced subgraphs have chromatic number at most $t$.

\begin{theorem}[Thomass\'{e}, Trotignon, Vu\v{s}kovi\'{c}; 8.1 and 8.2 in \cite{tripaper}]\label{thm:trithm}
For all $t\in \poi$, every graph $G\in \mca{B}_{t}$ satisfies $\chi(G)\leq p^p$ with $p=O(\omega(G)^2+t)$.
\end{theorem}

One motivation for \ref{thm:trithm} was a conjecture \cite{nesetril} that \textsl{every} hereditary class in which triangle-free graphs have bounded chromatic number is $\chi$-bounded. This was later refuted by Carbonero, Hompe, Moore, and Spirkl \cite{espertdis1}. Soon after, Bria\'{n}ski, Davies, and Walczak \cite{espertdis2} adapted the method of \cite{espertdis1} and disproved another conjecture, due to Esperet \cite{esperet}, that every $\chi$-bounded hereditary class is polynomially $\chi$-bounded. 

Recently, Chudnovsky, Cook, Davies, and Oum \cite{polypaper} proved that Esperet's conjecture holds in the class of bull-free graphs (they also proved the same for several other hereditary classes, yet declared bull-free graphs as ``the most difficult case'' \cite{polypaper}). In fact, they proved the following:

\begin{theorem}[Chudnovsky, Cook, Davies, Oum; see 1.2(v) in \cite{polypaper}]\label{thm:polythm}
For all $t\in \poi$, the class $\mca{B}_{t}$ is polynomially $\chi$-bounded.
\end{theorem}

The proofs of both \ref{thm:trithm} and \ref{thm:polythm} rely on the structure theorem for bull-free graphs. This is a difficult result due to Chudnovsky, whose statement alone requires several pages of definitions, and whose proof is spread across a series of four papers \cite{bullfree1,bullfree23,bullfree2,bullfree3}. The proof of \ref{thm:polythm}, in particular, takes over 10 pages and uses several results from \cite{bullfree1,bullfree23,bullfree2,bullfree3} on the structure of bull-free graphs and ``trigraphs.'' As for the bound, it is mentioned in \cite{polypaper} that a careful analysis of their proof gives $\chi(G)\leq \omega(G)^{O(t)}$ for all $G\in \mca{B}_t$.

Our main result is the following strengthening of \ref{thm:polythm}:

\begin{theorem}\label{thm:main}
We have $\chi(G)\leq \omega(G)^{4\log t+13}$ for all $t\in \poi$ and every graph $G\in \mca{B}_t$.
\end{theorem}

This has two advantages. First, the proof is a single page long. It entirely bypasses the material from \cite{bullfree1,bullfree23,bullfree2,bullfree3} and combines only one result of Chudnovsky and Safra \cite{bullfreeeh} (which has a proof no longer than five pages) with a ``layering argument.'' We remark that it is quite unusual for a layering proof to yield polynomial bounds. Second, not only does \ref{thm:main} give an exponential improvement in the degree of the polynomial bound claimed in \cite{polypaper}, it turns out that logarithmic degree in $t$ is asymptotically best possible:

\begin{theorem}\label{thm:lowerbound}
For all $t\in \poi$, every polynomial $\chi$-bounding function for $\mca{B}_t$ has degree larger than $\frac{1}{2}\log t$.
\end{theorem}

We will prove \ref{thm:main} and \ref{thm:lowerbound} in Sections~\ref{sec:main} and \ref{sec:lowerbound}, respectively.

\section{The upper bound}\label{sec:main}

Let $G$ be a graph. For $X\subseteq V(G)$, we use $X$ to denote both the set $X$ and the induced subgraph of $G$ with vertex set $X$. For $v\in V(G)$, we write $N_X(v)$ for the set of all neighbors of $v$ in $X$. We say that $X$ is a \textit{homogeneous set in $G$} if $1<|X|<|V(G)|$ and for every $v\in V(G)\setminus X$, either $N_X(v)=X$ or $N_X(v)=\varnothing$. We say that $G$ is \textit{prime} if there is no homogeneous set in $G$. Recall that a graph $G$ is \textit{perfect} if $\chi(H)\leq \omega(H)$ for every induced subgraph $H$ of $G$. We say that $G$ is \textit{N-perfect} if for every $v\in V(G)$, either $N_G(v)$ or $G\setminus N_G(v)$ is perfect. Observe that every induced subgraph of an N-perfect graph is also N-perfect.

We only use one result on the structure of bull-free graphs, which has a fairly short proof:

\begin{theorem}[Chudnovsky and Safra; see 1.4 and 4.3 in \cite{bullfreeeh}]\label{thm:CSbullfree}
Every prime bull-free graph is N-perfect.
\end{theorem}

We also need the following. A similar result from \cite{substituition} was also used in \cite{polypaper} to prove \ref{thm:polythm}.

\begin{theorem}[Bourneuf and Thomass\'{e}; see 6.1 in \cite{twwpoly}]\label{thm:subs}
Let $\mca{C}$ be a hereditary class and let $k\in \poi$ such that $\chi(G)\leq \omega(G)^k$ for every prime graph $G\in \mca{C}$. Then $\chi(G)\leq \omega(G)^{2k+3}$ for every graph $G\in \mca{C}$.
\end{theorem}

We are now ready to prove our main result:

\begin{proof}[Proof of \ref{thm:main}]
We may assume that $t\geq 2$ (because $\chi(G)=\omega(G)=1$ for all $G\in \mca{B}_{1}$). By \ref{thm:CSbullfree} and \ref{thm:subs}, it suffices to show that $\chi(G)\leq \omega(G)^{2\log t+5}$ for every N-perfect graph $G\in \mca{B}_t$. The proof is by induction on $\omega(G)$, and we may assume that $\omega(G)\geq 2$ and $G$ is connected. Suppose that $G\setminus N_G(u)$ is perfect for some $u\in V(G)$. Since $\omega(N_G(u))\leq \omega(G)-1$, it follows from the inductive hypothesis that $\chi(N_G(u))\leq (\omega(G)-1)^{2\log t+5}$. But now $\chi(G)\leq (\omega(G)-1)^{2\log t+5}+\omega(G)\leq \omega(G)^{2\log t+5}$, as desired.

Henceforth, assume that $N_G(u)$ is perfect for all $u\in V(G)$. Let $v\in V(G)$ be fixed. For each $r\in \poi\cup \{0\}$, let $L_r$ be the set of all vertices at distance exactly $r$ from $v$ in $G$. Note that $L_0=\{v\}$ and $L_1=N_G(v)$ (and $L_r=\varnothing$ for all sufficiently large $r$).

\sta{\label{st:primeatdistance} Let $r\in \poi$, let $H$ be a prime induced subgraph of $L_r$, and let $x\in L_{r-1}$. Then either $N_H(x)=H$ or $N_H(x)$ is a stable set.}

If $r=1$, then $x=v$ and $N_H(x)=N_H(v)=H$. Assume that $r\geq 2$. Then $x$ has a neighbor $y\in L_{r-2}$. Suppose for a contradiction that $|N_H(x)|<|H|$ and $N_H(x)$ is not a stable set. Then $N_H(x)$ has a component $K$ with $1<|K|\leq |N_H(x)|<|H|$. Since $K$ is not a homogeneous set in $H$, there is a vertex in $H\setminus K$ with both a neighbor and a non-neighbor in $K$. Since $K$ is a component of $N_H(x)$, there is a vertex $w\in H\setminus N_H(x)$ as well as distinct and adjacent vertices $u,u'\in K$ such that $uw\in E(G)$ and $u'w\notin E(G)$. But now $\{u,u',w,x,y\}$ induces a bull in $G$, a contradiction. This proves \eqref{st:primeatdistance}.

\sta{\label{st:distancechi} For every $r\in \poi\cup \{0\}$, we have $\chi(L_r)\leq \omega(G)^{2\log t+4}$.}

This is immediate for $r\in \{0,1\}$ because $L_r$ is perfect. Assume that $r\geq 2$. For every $a\in \{1,\ldots, r\}$, we define $S^a_1,\ldots, S^a_{\omega(G)}\subseteq L_a$ recursively, as follows. For $a=1$, let $S^1_1,\ldots, S^1_{\omega(G)}\subseteq L_1$ be (possibly empty) stable sets with union $L_1$, which exist since $L_1=N_G(v)$ is perfect. For $a\in \{2,\ldots, r\}$, having defined $S^{a-1}_1,\ldots, S^{a-1}_{\omega(G)}\subseteq L_{a-1}$, let $S^a_b$, for each $b\in \{1,\ldots, \omega(G)\}$, be the set of all vertices in $L_a$ with a neighbor in $S^{a-1}_b$. Note that for every $a\in \{1,\ldots, r\}$, the sets $S^a_1,\ldots, S^a_{\omega(G)}\subseteq L_a$ have union $L_a$. So in order to prove \eqref{st:distancechi}, it is enough to show that $\chi(S^r_b)\leq \omega(G)^{2\log t+3}$ for every $b\in \{1,\ldots, \omega(G)\}$. To that end, by \ref{thm:subs}, it suffices to show that every prime induced subgraph $H$ of $S^r_b\subseteq L_r$ satisfies $\chi(H)\leq \omega(H)^{\log t}$. If $\omega(H)\leq 1$, then $\chi(H)\leq 1$. So assume that $\omega(H)\geq 2$. Also, if there is a vertex $x\in L_{r-1}$ for which $N_H(x)=H$, then $H$ is perfect and we are done (because $t\geq 2$). Thus, by \eqref{st:primeatdistance}, we may assume that $N_H(x)$ is stable for every $x\in L_{r-1}$. We claim that $\omega(H)=2$. Suppose not. Then there is a triangle $\{u_1,u_2,u_3\}$ in $H$. For each $i\in \{1,2,3\}$, choose a neighbor $x_i\in S^{r-1}_b\subseteq L_{r-1}$ of $u_i$; note that $u_i$ is the only neighbor of $x_i$ in $\{u_1,u_2,u_3\}$ because $N_H(x_i)$ is stable. In particular, $x_1,x_2,x_3$ are pairwise distinct. If there are distinct $i,j\in \{1,2,3\}$ for which $x_ix_j\notin E(G)$, then $\{u_1,u_2,u_3,x_i,x_j\}$ induces a bull in $G$, a contradiction. So $\{x_1,x_2,x_3\}\subseteq S^{r-1}_b\subseteq L_{r-1}$ is a triangle in $G$. Since $S^1_b$ is stable, it follows that $r\geq 3$, which in turn implies that $x_1$ has a neighbor $y\in L_{r-2}$, and $y$ has a neighbor $z\in L_{r-3}$. Now, if $x_2y,x_3y\notin E(G)$, then $\{u_2,x_1,x_2,x_3,y\}$ induces a bull in $G$, and if $x_iy\in E(G)$ for some $i\in \{2,3\}$, then $\{u_1,x_1,x_i,y,z\}$ induces a bull in $G$, a contradiction. The claim that $\omega(H)=2$ follows. Since $G\in \mca{B}_t$ and $\omega(H)=2$, it follows that $\chi(H)\leq t=\omega(H)^{\log t}$. This proves \eqref{st:distancechi}.
\medskip

Note that $\chi(G)\leq 2\max_{r\geq 0}\chi(L_r)$ because $G$ is connected. Hence, by \eqref{st:distancechi} and since $\omega(G)\geq 2$, we have $\chi(G)\leq \omega(G)^{2\log t+5}$.
\end{proof}

\section{The lower bound}\label{sec:lowerbound}

The proof of \ref{thm:lowerbound} uses the ``Mycielski construction'' \cite{mycielski}, which yields a sequence $\{M_n\}_{n\geq 0}$ of graphs such that for each $n\in \poi$, we have $\omega(M_n)=2$ and $\chi(M_n)=n+1$ (the construction is well-known and we do not need the definition). There is also a beautiful recursion for the ``fractional chromatic number'' of the Mycielski graphs \cite{mycielskifrac} (again, we do not need the definition of the fractional chromatic number). For each $n\in \poi$, let $\phi_n$ be the fractional chromatic number of $M_n$.

\begin{theorem}[Larsen, Propp, Ullman \cite{mycielskifrac}]\label{thm:mycielskifrac}
We have $\phi_1=2$ and $\phi_{n+1}=\phi_{n}+\phi_n^{-1}$ for every $n\in \poi$.
\end{theorem}

\begin{corollary}\label{cor:frac}
For every $n\in \poi$, we have $\phi_n\geq \sqrt{2(n+1)}$.
\end{corollary}

\begin{proof}
This is clear for $n=1$. Assume that $n\geq 2$. Then by \ref{thm:mycielskifrac}, we have
\[
\phi_n^2=\phi_1^2+\sum_{k=1}^{n-1}\left(\phi^2_{k+1}-\phi_k^2\right)
=4+\sum_{k=1}^{n-1}\left(2+\phi_k^{-2}\right)
\geq 4+2(n-1)=2(n+1).
\]
\end{proof}

Let $G_1$ and $G_2$ be graphs with disjoint non-empty vertex sets, and let $x\in V(G_1)$. A graph $G$ is obtained by \textit{substituting $G_2$ for $x$ in $G_1$} if $V(G)=(V(G_1)\setminus \{x\})\cup V(G_2)$, and the following hold:
\begin{itemize}
    \item For all $u,v\in V(G_1)\setminus \{x\}$, we have $uv\in E(G)$ if and only if $uv\in E(G_1)$.
    \item For all $u,v\in V(G_2)$, we have $uv\in E(G)$ if and only if $uv\in E(G_2)$.
    \item Each $v\in N_{G_1}(x)$ is adjacent in $G$ to all vertices in $V(G_2)$.
    \item Each $v\in V(G_1)\setminus (N_{G_1}(x)\cup \{x\})$ is non-adjacent in $G$ to all vertices in $V(G_2)$.
\end{itemize}
In particular, if $|V(G_1)|=1$, then $G=G_2$, if $|V(G_2)|=1$, then $G=G_1$, and if $|V(G_1)|,|V(G_2)|>1$, then $V(G_2)$ is a homogeneous set in $G$. Given a graph class $\mca{C}$, we write $\mca{C}^*$ for the closure of $\mca{C}$ under disjoint union and substitution. Observe that if $\mca{C}$ is hereditary, then so is $\mca{C}^*$.

\begin{theorem}[Chudnovsky, Penev, Scott, Trotignon; see 2.4 in \cite{substituition}]\label{thm:frac}
Let $d\in \poi$ and let $H$ be a triangle-free graph of fractional chromatic number larger than $2^d$. Let $\mca{C}$ be the union of the class of all induced subgraphs of $H$ and the class of all complete graphs. Then every polynomial $\chi$-bounding function for $\mca{C}^*$ has degree larger than $d$.
\end{theorem}

We now prove \ref{thm:lowerbound}:

\begin{proof}[Proof of \ref{thm:lowerbound}]
The result is immediate for $t=1$. Assume that $t\geq 2$. Let $f$ be a polynomial $\chi$-bounding function for $\mca{B}_t$ of degree $\delta\in \poi$. Let $\mca{I}$ be the class of all induced subgraphs of $M_{t-1}$ and let $\mca{C}$ be the union of $\mca{I}$ and the class of all complete graphs.

Observe that $\mca{C}$ is hereditary, and so $\mca{C}^*$ is hereditary. Also, every graph in $\mca{C}$ is bull-free. Since a bull is connected and has no homogeneous set, it follows that every graph in $\mca{C}^*$ is bull-free. Moreover, since $\mca{I}$ is exactly the class of all triangle-free graphs in $\mca{C}$, it follows that the class of all triangle-free graphs in $\mca{C}^*$ is exactly the closure of $\mca{I}$ under disjoint union and substitution of stable sets for vertices. Since $\chi(M_{t-1})=t\geq 2$, it follows that $\chi(G)\leq t$ for every triangle-free graph $G\in \mca{C}^*$. We deduce that $\mca{C}^*\subseteq \mca{B}_t$, and so $f$ is a polynomial $\chi$-bounding function for $\mca{C}^*$.

On the other hand, by \ref{cor:frac}, we have $\phi_{t-1}\geq \sqrt{2t}
=2^{\frac{1}{2}(\log t +1)}
>2^{\lfloor \frac{1}{2}\log t\rfloor}$. Hence, by \ref{thm:frac}, we have $\delta>\lfloor \frac{1}{2}\log t\rfloor$, and thus $\delta>\frac{1}{2}\log t$ since $\delta$ is an integer.
\end{proof}

\section{Acknowledgement}
Our thanks to Sophie Spirkl for helpful feedback.

\bibliographystyle{plain}
\bibliography{ref}

\end{document}